\documentclass{amsart}

\usepackage[utf8]{inputenc}

\usepackage{amsmath, amsthm, amssymb, cite, bm,enumitem,bbm,xcolor,mathtools}

\advance \topmargin by -\headheight
\advance \topmargin by -\headsep
     
\evensidemargin \oddsidemargin
\newtheorem{theorem}{Theorem}[]
\newtheorem{lemma}[theorem]{Lemma}
\newtheorem{proposition}[theorem]{Proposition}

\newtheorem{conjecture}[theorem]{Conjecture}

\theoremstyle{definition}
\newtheorem{definition}[theorem]{Definition}

\theoremstyle{remark}

\numberwithin{equation}{section}

\newcommand{\diff}[1]{\hbox{d}#1}

\DeclareMathOperator{\Nm}{N}
\DeclareMathOperator{\dnm}{dnm}

\newcommand{\NN}{\mathbb{N}}
\newcommand{\ZZ}{\mathbb{Z}}

\newcommand{\RR}{\mathbb{R}}
\newcommand{\CC}{\mathbb{C}}

\newcommand{\twovect}[2]{\left( \begin{array}{c} #1 \\ #2 \end{arraIntroductiony} \right)}

\newcommand{\cA}{\mathcal{A}}

\newcommand{\cO}{\mathcal{O}}
\newcommand{\cP}{\mathcal{P}}

\newcommand{\fa}{\mathfrak{a}}
\newcommand{\fb}{\mathfrak{b}}

\newcommand{\fd}{\mathfrak{d}}
\newcommand{\fm}{\mathfrak{m}}
\newcommand{\fn}{\mathfrak{n}}
\newcommand{\fp}{\mathfrak{p}}
\newcommand{\fq}{\mathfrak{q}}
\newcommand{\fr}{\mathfrak{r}}
\newcommand{\fL}{\mathfrak{L}}
\newcommand{\fM}{\mathfrak{M}}
\newcommand{\fN}{\mathfrak{N}}
\newcommand{\fP}{\mathfrak{P}}

\newcommand{\x}{\bm x}

\newcommand{\aand}{\quad \hbox{and} \quad}

\newcommand{\rset}[2]{\left\{ #1 \ \left| \ #2  \right. \right\}}

\newcommand{\defiso}[5]{\[\begin{array}{rrcl} #1: & #2 & \longrightarrow & #3 \\ & #4 & \longmapsto & #5 \end{array}\]}

\newcommand{\PSI}{\underline{\psi}}
\DeclareMathOperator{\lcm}{lcm}
\DeclareMathOperator{\Bx}{Bx}
\newcommand{\fD}{\mathfrak{D}}
\newcommand{\ff}{\mathfrak{f}}
\newcommand{\fg}{\mathfrak{g}}
\newcommand{\fh}{\mathfrak{h}}
\newcommand{\fk}{\mathfrak{k}}
\newcommand{\fs}{\mathfrak{s}}
\newcommand{\fl}{\mathfrak{l}}
\newcommand{\fv}{\mathfrak{v}}

\title[Vaaler's theorem in number fields]{Vaaler's theorem in number fields}
\author{Matthew Palmer}
\thanks{The author was supported by the Knut and Alice Wallenberg Foundation}

\begin{document}

\maketitle

\begin{abstract}
In Diophantine approximation, Vaaler's theorem was an important partial result towards the Duffin--Schaeffer conjecture, which was open for almost eighty years before it was recently proven by Koukoulopoulos and Maynard. A version of this result was previously proven to also hold in imaginary quadratic fields: in this paper, we establish a version of Vaaler's theorem in general number fields.
\end{abstract}

\tableofcontents

\section*{Notation}
In this paper, \(\varphi\) represents the classical Euler totient function. Our convention for the set \(\NN\) of natural numbers is that it should not include \(0\). 

\section{Introduction}
Consider a function \(\psi : \NN \to \RR_{\geq 0}\), and define the set of \(\psi\)-approximable numbers
\[
A'(\psi) = \rset{x \in [0,1]}{\left| x - \frac{a}{n} \right| \leq \frac{\psi(n)}{n} \hbox{ for infinitely many } a \in \ZZ, n \in \NN \hbox{ with } (a,n) = 1 }. \]

In a 1941 paper (see \cite{DuffinSchaeffer1941}), Duffin and Schaeffer proved the following theorem, building on work of Khinchin and others (see \cite{Dirichlet1842}, \cite{Hermite1850}, \cite{Hurwitz1891}, \cite{Khinchin1924}):

\begin{theorem}[Duffin \& Schaeffer, 1941]
\label{thm:classicalDS}
Let \(\psi\) be such that
\[
\sum_{n \in \NN} \frac{\varphi(n) \psi(n)}{n} = \infty \quad \hbox{and} \quad \limsup_{N \to \infty} \frac{\sum_{n \leq N} \frac{\varphi(n) \psi(n)}{n}}{\sum_{n \leq N} \psi(n)} > 0.\]

Then the set \(A'(\psi)\) has Lebesgue measure \(1\).
\end{theorem}

In the same paper, they suggested that the second condition might be unnecessary, which over time became known as the \textit{Duffin--Schaeffer conjecture}:

\begin{conjecture}[Duffin--Schaeffer conjecture]
Let \(\psi\) be such that
\[
\sum_{n \in \NN} \frac{\varphi(n) \psi(n)}{n} = \infty. \]

Then the set \(A'(\psi)\) has Lebesgue measure \(1\).
\end{conjecture}

This conjecture remained open for almost eighty years, until it was proven by Koukoulopoulos and Maynard in 2019 (see \cite{KoukoulopoulosMaynard2020}). However, in this paper we will be more interested in the work that predates this achievement, in which various partial results towards the conjecture were proven.

The particular result we will be interested in is a result from 1978, due to Vaaler. Previously, Erd\H{o}s had proven the following result (see \cite{Erdos1970}):

\begin{theorem}[Erd\H{o}s, 1970]
Suppose that there exists \(\varepsilon\) such that for every \(n \in \NN\), the value of the function \(\psi\) at \(n\) is either \(\frac{\varepsilon}{n}\) or \(0\). Then the Duffin--Schaeffer conjecture holds for \(\psi\).
\end{theorem}

Vaaler then went on to strengthen this (see \cite{Vaaler1978}), and it is this theorem we would like to generalise:

\begin{theorem}[Vaaler, 1978]
\label{thm:classicalVaaler}
Suppose \(\psi(n) = O(\frac{1}{n})\). Then the Duffin--Schaeffer conjecture holds for \(\psi\).
\end{theorem}

In previous work (see \cite{Palmer2020}), the author developed a version of Diophantine approximation in general number fields, expanding on work of Nakada and Wagner (see \cite{NakadaWagner1991}) in imaginary quadratic fields, and proved a version of Theorem~\ref{thm:classicalDS} in this setup. In this paper, we prove a version of Theorem~\ref{thm:classicalVaaler} in the same setup. (A version in imaginary quadratic fields, using Nakada and Wagner's setup, was previously proven by Chen in \cite{Chen2015}.) We take a slightly more modern approach to that in \cite{Vaaler1978}, using methods found in \cite{PollingtonVaughan1990} and in \cite{HarmanBook}.

In \S\ref{sec:setup}, we reiterate the setup developed in \cite{Palmer2020}, and use this to state our main result. In \S\ref{sec:mainproof}, we provide a proof of the main theorem, modulo some major results which we defer to later sections. In \S\ref{sec:Pmn}, we prove a version of Pollington and Vaughan's \(P(m,n)\) bound from \cite{PollingtonVaughan1990}, modulo a certain sieve-theoretic result, and in \S\ref{sec:sievetheory}, we establish this sieve-theoretic result. Finally, in \S\ref{sec:NFgfct}, we define the natural equivalent of Erd\H{o}s's \(g\)-function in number fields, and establish the necessary properties of this function.


\section{Setup and statement of result}
\label{sec:setup}
This setup was developed in \cite{Palmer2020}; we reproduce it here for completeness.

Let \(K\) be a number field of degree \(n\). Let \(\cO_K\) denote its ring of integers, and let \(I_K\) denote the semigroup of integral ideals of \(\cO_K\). We define the norm function and the Euler \(\Phi\)-function on \(I_K\) by
\defiso{\Nm}{I_K}{\NN}{\fn}{\#(\cO_K / \fn \cO_K)}
and
\defiso{\Phi}{I_K}{\NN}{\fn}{\#((\cO_K / \fn \cO_K)^{\times}).}

The symbol \(\fp\) will always denote a prime ideal of \(\cO_K\), and hence a sum of the form
\[
\sum_{\fp} *
\]
should be understood as a sum over the prime ideals of \(\cO_K\).

Here we cite the following results for certain sums and products relating to prime ideals, which are analogues of the classical results due to Mertens (see \cite{Mertens1874}), and which we will refer to frequently. (These are Lemmas 2.3 and 2.4 and Theorem 2 respectively in \cite{Rosen1999}.)

\begin{theorem}\label{thm:Merten1}
Let \(K\) be an algebraic number field. Then
\[
\sum_{\Nm(\fp) < X} \frac{\log(\Nm(\fp))}{\Nm(\fp)} = \log X + O_K(1).\]
\end{theorem}

\begin{theorem}\label{thm:Merten2}
Let \(K\) be an algebraic number field. Then there is a constant \(B_K\) such that
\[
\sum_{\Nm(\fp) \leq X} \frac{1}{\Nm(\fp)} = \log \log X + B_K + O_K\left(\frac{1}{\log X} \right).\]
\end{theorem}

\begin{theorem}\label{thm:Merten3}
Let \(K\) be an algebraic number field, and let \(\alpha_K\) be the residue of the associated Dedekind zeta function \(\zeta_K(s)\) at \(s=1\). Then
\[
\prod_{\Nm(\fp) \leq X} \left( 1 - \frac{1}{\Nm(\fp)} \right)^{-1} = e^{\gamma} \alpha_K \log X + O_K(1).\]
\end{theorem}

Suppose that our number field \(K\) has \(s\) real embeddings and \(t\) pairs of complex embeddings (where \(s + 2t = n\)); we denote these by \(\sigma_1, \ldots, \sigma_s\) and \(\tau_1,\ldots,\tau_t\) respectively. We denote the set of all embeddings of \(K\) by \(\Sigma\), and denote a generic embedding by \(\rho\).

We also define \(|\cdot|_{\RR}\) to be the standard real absolute value, and \(|\cdot|_{\CC}\) to be the square of the standard complex absolute value. Then we define \(|\cdot|_{\rho}\) to be either \(|\cdot|_{\RR}\) if \(\rho\) is real, or \(|\cdot|_{\CC}\) if \(\rho\) is complex. When taking the absolute value of something explicitly involving \(\rho\), we will often just assume that the absolute value is with respect to \(\rho\), and suppress the subscript for brevity (writing \( |\rho(\alpha)|\) in place of \( |\rho(\alpha)|_{\rho}\), for example).

For any element \(\gamma \in K\), we define the norm \(\Nm(\gamma)\) of \(\gamma\) by
\[
\Nm(\gamma) = \prod_{\sigma \hbox{\scriptsize \ real}} \sigma(\gamma) \prod_{\tau \hbox{\scriptsize \ complex}} \tau(\gamma) \overline{\tau(\gamma)}. \]

(Note this agrees with the ideal norm up to sign: that is to say, the absolute value of the norm of an element \(\gamma\) is the same as the ideal norm of the principal ideal generated by \(\gamma\).) We identify each element of \(K\) with an element of \(\RR^s \times \CC^t\) by embedding it into each of its completions. That is to say, we define a map \(\iota : K \to \RR^s \times \CC^t\) by
\[
\iota(\alpha) = (\sigma_1(\alpha), \ldots, \sigma_s(\alpha), \tau_1(\alpha), \ldots, \tau_t(\alpha)). \]

The image \(\iota(\cO_K)\) of \(\cO_K\) under this map forms a lattice in \(\RR^s \times \CC^t\). If we let \(\beta_1,\ldots,\beta_n\) be an integral basis for \(\cO_K\), we have a fundamental domain for this lattice given by
\[
D_K = \Big\{ t_1 \iota(\beta_1) + \cdots + t_n \iota(\beta_n) \ \Big| \ t_1,\ldots,t_n \in [0,1] \Big\}.\]

We have a measure \(\lambda\) on \(D_K\) induced by the Lebesgue measure on \(\RR^s \times \CC^t\).

As a result of our diagonal embedding \(\iota\) of \(K\) into \(\RR^s \times \CC^t\), we can index the components of an element \(\x \in \RR^s \times \CC^t\) by the embeddings of \(K\). That is, we can write
\[
\x = (x_1,\ldots,x_s,x_{s+1},\ldots,x_{s+t}) = (x_{\sigma_1},\ldots,x_{\sigma_s},x_{\tau_1},\ldots,x_{\tau_t}). \]

Then for any embedding \(\rho \in \Sigma\), we can refer to the \(\rho\)-coordinate \(x_{\rho}\) of an element \(\x \in \RR^s \times \CC^t\).

By Dirichlet's unit theorem, the group of units of \(\cO_K\) has rank \(\ell = s+t-1\). That is to say, there exist a set of multiplicatively independent elements \(\{u_1,\ldots,u_{\ell}\} \subset \cO_K^{\times}\) such that any element \(u \in \cO_K^{\times}\) can be written as
\[
u = \zeta u_1^{n_1} \cdots u_{\ell}^{n_{\ell}},
\]
where \(\zeta\) is some root of unity in \(\cO_K^{\times}\). We call such a set \(\{u_i\}\) a system of fundamental units of \(K\).

For each embedding \(\rho \in \Sigma\), we choose a function \(\psi_{\rho} : I_K \to \RR_{\geq 0}\). We combine these into one function \(\PSI\) by defining
\defiso{\PSI}{I_K}{\RR_{\geq 0}^{s+t}}{\fn}{\bigoplus_{\rho \in \Sigma} \psi_{\rho}(\fn).}

We also define a function \(\Psi : I_K \to \RR_{\geq 0}\) by
\[
\Psi(\fn) = \prod_{\rho \in \Sigma} \psi_{\rho}(\fn).\]

For any element \(\gamma \in K\), we have a unique way of writing \((\gamma) = \frac{\fa}{\fn}\) with \(\fa,\fn \in I_K\) and \((\fa,\fn) = 1\). Then we write \(\dnm \gamma = \fn\).

For \(\x \in \RR^s \times \CC^t\), we say that \(\gamma \in K\) is a \(\PSI\)-good approximation to \(\x\) if we have
\[
|x_{\rho} - \rho(\gamma)| \leq \psi_{\rho}(\dnm(\gamma))
\]
for each \(\rho \in \Sigma\). We then define a set \(\cA(\PSI)\) by
\[
\cA({\PSI}) = \rset{\x \in D_K}{\begin{array}{c}\hbox{there exist infinitely many $\gamma \in K$ such} \\ \hbox{that $\gamma$ is a $\PSI$-good approximation to \(\x\)} \end{array}}.\]

We also need to introduce the following technical condition: we call the function \(\PSI\) \textit{balanced} if for all subsets \(\Sigma' \subset \Sigma\) with \(2 |\Sigma'| = |\Sigma|\), we have that
\[
\frac{\displaystyle \prod_{\rho \in \Sigma'} \psi_{\rho}(\fn)}{\displaystyle \prod_{\rho \in \Sigma \backslash \Sigma'} \psi_{\rho}(\fn)} \ll 1.\]

For example, if \(K\) is a real quadratic field, then we are requiring that \(\frac{\psi_1}{\psi_2}\) and \(\frac{\psi_2}{\psi_1}\) are bounded (or equivalently that \(\frac{\psi_1}{\psi_2}\) is bounded away from \(0\) and \(\infty\)). Note that any function on a number field with an \textit{odd} number of embeddings is \textit{trivially} balanced.

With this setup done, our version of Vaaler's theorem is as follows:

\begin{theorem}\label{thm:maintheorem}
If \(\PSI\) is a balanced function such that
\[
\sum_{\fn \in I_K} \Phi(\fn) \Psi(\fn) = \infty
\]
and
\begin{equation}
\label{eq:llcondition}
\psi_{\rho}(\fn) \ll \frac{1}{\Nm(\fn)^{\frac{2}{s+t}}},
\end{equation}
for each \(\rho \in \Sigma\), then \(\cA(\PSI)\) has measure \(\lambda(D_K)\).
\end{theorem}

In the next section, we will provide a proof of this theorem, modulo some major results which we defer to later sections.


\section{Proof of the main theorem}
\label{sec:mainproof}

In this section, we provide a proof of Theorem~\ref{thm:maintheorem}, modulo some major results which we defer to later sections.

The first step of the proof is very standard; for each \(\fn \in I_K\), we define a set \(\cA_{\fn}(\PSI)\) by
\[
\cA_{\fn}({\PSI}) = \rset{\x \in D_K}{\begin{array}{c}\hbox{there is some $\gamma \in K$ with $\dnm \gamma = \fn$} \\ \hbox{such that $\gamma$ is a $\PSI$-good approximation to \(\x\)} \end{array}}.\]

(Given we tend to work with a fixed \(\PSI\) at any given time, we will usually suppress \(\PSI\) in the notation and just write \(\cA_{\fn}\).) Then taking the same approach as in \S4 of \cite{Palmer2020}, we find that
\[
\lambda(\cA(\PSI)) \geq \limsup_{R \in \NN} \left( \sum_{\substack{\fn \in I_K \\ \Nm(\fn) \leq R}} \lambda(\cA_{\fn}) \right)^2 \left( \sum_{\substack{\fm,\fn \in I_K \\ \Nm(\fm),\Nm(\fn) \leq R}} \lambda(\cA_{\fm} \cap \cA_{\fn}) \right)^{-1}.
\]

By the zero-one law (see Theorem 2.1 of \cite{Palmer2020}), we only need to show that \(\lambda(\cA(\PSI)) > 0\) to show that we have full measure. So we will show that the right-hand side of the above inequality is greater than \(0\). We will do this by showing that
\begin{equation}
\label{eq:aim1}
\sum_{\substack{\fm,\fn \in I_K \\ \Nm(\fm) < \Nm(\fn) \leq R}} \lambda(\cA_{\fm} \cap \cA_{\fn}) \ll \left( \sum_{\substack{\fn \in I_K \\ \Nm(\fn) \leq R}} \lambda(\cA_{\fn}) \right)^2. \end{equation}

One of the first steps (and one of the key ingredients of the proof) is the following result, a version of Pollington and Vaughan's \(P(m,n)\) bound from \cite{PollingtonVaughan1990}:

\begin{proposition}
\label{prop:Pmnbound}
For \(\fm \neq \fn \in I_K\), we have
\[
\lambda(\cA_{\fm} \cap \cA_{\fn}) \ll P(\fm,\fn) \lambda(\cA_{\fm}) \lambda(\cA_{\fn}),
\]
where
\[
P(\fm,\fn) = \prod_{\substack{\fp | \frac{\fm\fn}{(\fm,\fn)^2} \\ \Nm(\fp) > D(\fm,\fn)}} \left(  1 - \frac{1}{\Nm(\fp)} \right)^{-1}
\]
and
\[
D(\fm,\fn) = \Nm\left(\tfrac{\fm \fn}{(\fm,\fn)}\right) \prod_{\rho \in \Sigma} \max \{ \psi_{\rho}(\fm), \psi_{\rho}(\fn) \}.\]
\end{proposition}

Proving this \(P(\fm,\fn)\) bound will be the focus of \S\ref{sec:Pmn} and \S\ref{sec:sievetheory}.

Next we will need to analyse this factor \(P(\fm,\fn)\). In order to do this, we need to introduce an analogue of Erd\H{o}s's \(g\)-function (the original version of which is defined on p430 of \cite{Erdos1970}):

\begin{definition}[\(g\)-function]
\label{def:gfct}
For an ideal \(\fn \in I_K\), we define \(g(\fn)\) to be the smallest natural number such that
\[
\sum_{\substack{\fp \mid \fn \\ \Nm(\fp) > g(\fn)}} \frac{1}{\Nm(\fp)} < \frac{1}{2}. \]
\end{definition}

In \S\ref{sec:NFgfct}, we will use a relationship between \(P(\fm,\fn)\) and \(g(\fn)\) to prove the following result:

\begin{lemma}\label{lem:overlapconditions}
Let 
\[
\tau = \tau(\fm,\fn) = \max \left\{ g \left( \tfrac{\fm}{(\fm,\fn)} \right), g \left( \tfrac{\fn}{(\fm,\fn)} \right) \right\}.\]

Then we have
\[
\lambda(\cA_{\fm} \cap \cA_{\fn}) \ll \left\{ \begin{array}{ll} \lambda(\cA_{\fm}) \lambda(\cA_{\fn}) \log(2\tau) & \hbox{if } \frac{1}{2^{s+t}} < D(\fm,\fn) < \tau \\ \lambda(\cA_{\fm}) \lambda(\cA_{\fn}) & \hbox{if } \tau \leq D(\fm,\fn) \end{array} \right. \]
\end{lemma}

(Note that if \(D(\fm,\fn) \leq \frac{1}{2^{s+t}}\), then the intersection is zero measure.)

Applying Lemma~\ref{lem:overlapconditions}, in view of our goal in \eqref{eq:aim1}, it will be enough to show that
\begin{equation}
\label{eq:sum1}
\sum_{\substack{\Nm(\fm), \Nm(\fn) \leq R \\ \frac{1}{2^{s+t}} < D(\fm,\fn) < \tau(\fm,\fn)}} \Psi(\fm) \Psi(\fn) \Phi(\fm) \Phi(\fn) \log 2\tau \ll \left( \sum_{\Nm(\fn) \leq R} \Phi(\fn) \Psi(\fn) \right)^2.\end{equation}

Two other properties of the \(g\)-function we will need are contained in the following lemmas:

\begin{lemma}
\label{lem:glemma1}
Let \(X \geq 1\) and \(v \in \NN\). Then
\[
\sum_{\substack{\Nm(\fn) \leq X \\ g(\fn) \geq v}} 1 \ll \frac{X}{v!}. \]
\end{lemma}

\begin{lemma}
\label{lem:divisorsum}
For \(v \in \NN\) and \(\fq \in I_K\), we have
\[
\sum_{\substack{\fd | \fq \\ g(\frac{\fq}{\fd}) \leq v}} \frac{1}{\Nm(\fd)} \ll \log 2v.\]
\end{lemma}

These lemmas will also be proved in \S\ref{sec:NFgfct}. Assuming these two results, we can now prove Theorem~\ref{thm:maintheorem}.

For brevity, write
\[
\tilde{\fm} = \frac{\fm}{(\fm,\fn)}, \quad \tilde{\fn} = \frac{\fn}{(\fm,\fn)}.\]

We start by noting that the left-hand side of \eqref{eq:sum1} is symmetric in \(\fm\) and \(\fn\). Then since we either have \(\tau = g(\tilde{\fm})\) or \(\tau = g(\tilde{\fn})\), we choose to only consider the former case: the latter can be dealt with by symmetry. 

We can split over possible values of \(\tau\), denoting these by \(T\):
\[
\sum_{T \in \NN} \log 2T \sum_{\substack{\Nm(\fm), \Nm(\fn) \leq R \\ \frac{1}{2^{s+t}} < D(\fm,\fn) < T \\ g(\tilde{\fm}) = T \\ g(\tilde{\fn}) \leq T}} \Psi(\fm) \Psi(\fn) \Phi(\fm) \Phi(\fn).\]

Now, for each pair \((\fm,\fn)\), let
\[
\Sigma_1(\fm,\fn) = \{ \rho \in \Sigma \ | \ \psi_{\rho}(\fm) \geq \psi_{\rho}(\fn) \}, \quad \Sigma_2(\fm,\fn) = \Sigma \backslash \Sigma_1(\fm,\fn).\]

Given there are only finitely many choices for \(\Sigma_1(\fm,\fn)\), it will be enough to prove that for each fixed subset \(\Sigma_1 \subset \Sigma\), we have that \eqref{eq:sum1} holds for the subsum where the \((\fm,\fn)\) satisfy \(\Sigma_1(\fm,\fn) = \Sigma_1\) (and hence \(\Sigma_2(\fm,\fn) = \Sigma \backslash \Sigma_1,\) which we naturally call \(\Sigma_2\)). In this case, we have that
\[
D(\fm,\fn) = \Nm \left( \frac{\fm \fn}{(\fm,\fn)} \right) \prod_{\rho \in \Sigma_1} \psi_{\rho}(\fm) \prod_{\rho \in \Sigma_2} \psi_{\rho}(\fn),
\]
and if we write
\[
\Psi_1(\fr) = \prod_{\rho \in \Sigma_1} \psi_{\rho}(\fr), \quad \Psi_2(\fr) = \prod_{\rho \in \Sigma_2} \psi_{\rho}(\fr),
\]
then we have
\[
D(\fm,\fn) = \Nm(\tilde{\fm}) \Nm(\fn) \Psi_1(\fm) \Psi_2(\fn).\]

Noting that \(\Phi(\fm) \leq \Nm(\fm)\), we find that our sum is bounded by
\[
\sum_{T \in \NN} \log 2T \sum_{\Nm(\fn) \leq R} \Phi(\fn) \Psi(\fn)  \sum_{\substack{\Nm(\fm) \leq R \\ \frac{1}{2^{s+t}} < \Nm(\tilde{\fm}) \Nm(\fn) \Psi_1(\fm) \Psi_2(\fn) < T \\ g(\tilde{\fm}) = T \\ g(\tilde{\fn}) \leq T}} \Psi(\fm) \Nm(\fm).\]

Consider just the innermost sum for fixed \(T\) and \(\fn\). We find this sum is equal to 
\[
\sum_{\substack{\fg \mid \fn \\ g(\fn/\fg) \leq T}} \sum_{\substack{\Nm(\fs) \leq \frac{R}{\Nm(\fg)} \\ \frac{1}{2^{s+t}} < \Nm(\fs) \Nm(\fn) \Psi_1(\fg \fs) \Psi_2(\fn) < T \\ g(\fs) = T}} \Psi(\fg \fs) \Nm(\fg \fs).\]

Next, for \(\rho \in \Sigma\) and \( h \in \ZZ\), we define
\[
\Xi_{\rho}(h) = \left\{ \fr \in I_K \ \left| \ \Nm(\fr)^{\frac{2}{s+t}} \psi_{\rho}(\fr) \in \left[ \left. 2^h, 2^{h+1} \right) \right. \right. \right\}.\]

By our assumption in \eqref{eq:llcondition}, we know that for each \(\rho\) there exists \(H_{\rho} \in \ZZ\) such that \(\Xi(h) = \emptyset\) for \(h \geq H_{\rho}\).

Let \(|\Sigma_1| = \varsigma\), and let  \(\rho_1,\rho_2,\ldots,\rho_{\varsigma}\) be the elements of \(\Sigma_1\). Then we once again tackle the innermost sum, and bound it above by
\begin{equation}
\label{eq:mostdecomposedsum}
\frac{1}{\Nm(\fg)} \sum_{h_1 \leq H_1} \cdots \sum_{h_{\varsigma} \leq H_{\varsigma}} 2^{h_1 + \cdots + h_{\varsigma} + \varsigma} \sum_{\substack{\frac{1}{2^{s+t} \Nm(\fn) \Psi_2(\fn)} < \Nm(\fs) \Psi_1(\fg \fs) < \frac{T}{\Nm(\fn) \Psi_2(\fn)} \\ g(\fs) = T \\ \fg \fs \in \Xi_1(h_1) \\ \cdots \\ \fg \fs \in \Xi_{\varsigma}(h_{\varsigma})}} \frac{1}{\Nm(\fs)}.
\end{equation}

We now aim to show that the multiple sum following \(\frac{1}{\Nm(\fg)}\) is \(\ll \frac{\log 2T}{T!}\).

Write \(\kappa = \frac{s+t-2\varsigma}{s+t}\), and assume for now that \(\kappa \neq 0\). Note that if 
\[
N(\fs) \Psi_1(\fg \fs) < \frac{T}{\Nm(\fn) \Psi_2(\fn)},
\]
then
\[
\Nm(\fs) < \frac{T}{\Nm(\fn) \Psi_2(\fn)} \prod_{1 \leq i \leq \varsigma} \frac{1}{\psi_{\rho_i}(\fg \fs)}.\]

Then since \(\fg \fs \in \Xi_i(h_i)\) for \(1 \leq i \leq \varsigma\), we have
\[
\frac{1}{\psi_{\rho_i}(\fg \fs)} \leq \frac{\Nm(\fg)^{\frac{2}{s+t}} \Nm(\fs)^{\frac{2}{s+t}}}{2^{h_i}},
\]
and hence
\[
\Nm (\fs)^{\kappa} < \frac{T \Nm(\fg)^{\frac{2\varsigma}{s+t}}}{\Nm(\fn) \Psi_2(\fn) 2^{h_1 + \cdots + h_{\varsigma}}}.\]

By the same line of argument but from below, we find that
\[
\Nm (\fs)^{\kappa} > \frac{\Nm(\fg)^{\frac{2\varsigma}{s+t}}}{\Nm(\fn) \Psi_2(\fn) 2^{h_1 + \cdots + h_{\varsigma} + s + t}}.\]

Writing
\[
X = \frac{\Nm(\fg)^{\frac{2\varsigma}{s+t}}}{\Nm(\fn) \Psi_2(\fn) 2^{h_1 + \cdots + h_{\varsigma} + s + t}}, \quad Y = \frac{T \Nm(\fg)^{\frac{2\varsigma}{s+t}}}{\Nm(\fn) \Psi_2(\fn) 2^{h_1 + \cdots + h_{\varsigma}}},
\]
we see that if \(\kappa > 0\), we have
\[
X^{\frac{1}{\kappa}} < \Nm(\fs) < Y^{\frac{1}{\kappa}},
\]
and if \(\kappa < 0\), we have
\[
Y^{\frac{1}{\kappa}} < \Nm(\fs) < X^{\frac{1}{\kappa}}.\]

So we can bound our sum by
\[
\sum_{h_1 \leq H_1}  \cdots \sum_{h_{\varsigma} \leq H_{\varsigma}} 2^{h_1 + \cdots + h_{\varsigma} + \varsigma} \sum_{\substack{X^{\frac{1}{\kappa}} < \Nm(\fs) < Y^{\frac{1}{\kappa}} \\ g(\fs) = T}} \frac{1}{\Nm(\fs)}
\]
if \(\kappa > 0\), and by
\[
\sum_{h_1 \leq H_1}  \cdots \sum_{h_{\varsigma} \leq H_{\varsigma}} 2^{h_1 + \cdots + h_{\varsigma} + \varsigma} \sum_{\substack{Y^{\frac{1}{\kappa}} < \Nm(\fs) < X^{\frac{1}{\kappa}} \\ g(\fs) = T}} \frac{1}{\Nm(\fs)}
\]
if \(\kappa < 0\). Once again, we tackle the innermost sum.

\begin{lemma}
For any \(0 < X  < Y\), we have
\[
\sum_{\substack{g(\fs) = T \\ X < \Nm(\fs) < Y}} \frac{1}{\Nm(\fs)} \ll \frac{\log \left( \frac{Y}{X} \right)}{T!}.\]\end{lemma}

\begin{proof}
Since all terms are positive, we get an upper bound if we replace the \(g(\fs) = T\) condition by the weaker condition \(g(\fs) \geq T\). We then write this sum as

\[
\sum_{X < \Nm(\fs) \leq Y} \frac{\chi(\fs)}{\Nm(\fs)}, \]
where \(\chi\) is the indicator function for the condition \(g(\fs) \geq T\). Now group our \(\fs\) by norm, giving
\[
\sum_{X < n \leq Y} \frac{1}{n} \sum_{\Nm(\fs) = n} \chi(\fs). \]

Now letting
\[
A(z) = \sum_{\substack{\Nm(\fr) \leq z \\ g(\fr) \geq T}} 1,
\]
we have by Abel summation that 
\[
\sum_{X < \Nm(\fs) \leq Y} \frac{\chi(\fs)}{\Nm(\fs)} = \frac{A(Y)}{Y} - \frac{A(X)}{X} + \int_X^Y \frac{A(t)}{t^2} \ \diff{t}. \]

Note that by Lemma~\ref{lem:glemma1}, we have that
\[
A(z) \ll \frac{z}{T!}.\]

So we have
\[
\frac{A(Y)}{Y} - \frac{A(X)}{X} \ll \frac{1}{T!},
\]
and
\[
\int_X^Y \frac{A(t)}{t^2} \ \diff{t} \ll \frac{1}{T!} \int_X^Y \frac{1}{t} \ \diff{t} =  \frac{\log \left( \frac{Y}{X} \right)}{T!}.\]

Hence 
\[
\sum_{\substack{g(\fs) = T \\ X < \Nm(\fs) < Y}} \frac{1}{\Nm(\fs)} \ll \frac{\log \left( \frac{Y}{X} \right)}{T!}
\]
as required. \end{proof}

Applying this lemma, we see that our sum is
\[
\ll \frac{1}{T!} \log \left( \left( \frac{Y}{X} \right)^{\frac{1}{\kappa}} \right) \quad \hbox{or} \quad \ll \frac{1}{T!} \log \left( \left( \frac{X}{Y} \right)^{\frac{1}{\kappa}} \right) = \frac{1}{T!} \log \left( \left( \frac{Y}{X} \right)^{-\frac{1}{\kappa}} \right)
\]
depending on the sign of \(\kappa\). However, in the first case we can pull out a factor of \(\frac{1}{\kappa}\), which is positive and at most \(s+t\); and in the second case we can pull out a factor of \(-\frac{1}{\kappa}\), which is also positive and also at most \(s+t\). So both factors can be disregarded as \(O(1)\), giving that in both cases we have that our sum is
\[
\ll \frac{1}{T!} \log \left( \frac{Y}{X} \right) \ll \frac{\log 2T}{T!},
\]
and hence that for \(\kappa \neq 0\), the sum given in \eqref{eq:mostdecomposedsum} is
\[
 \ll \frac{\log 2T}{T!} \prod_{i=1}^{\varsigma} \left( \sum_{h_i \leq H_i}  2^{h_i} \right) \ll \frac{\log 2T}{T!}
 \]
 as claimed.

 Now suppose that \(\kappa = 0\): that is to say, that \(2 \varsigma = s+t\). Then by our assumption that \(\PSI\) is balanced, we have that
 \[
 \frac{\Psi_1(\fr)}{\Psi_2(\fr)} = \frac{\prod_{\rho \in \Sigma_1} \psi_{\rho}(\fr)}{\prod_{\rho \in \Sigma_2} \psi_{\rho}(\fr)} \asymp 1,
 \]
 and hence there exists a constant \(C > 1\) such that
 \[
 \frac{\Psi_1(\fr)}{\Psi_2(\fr)} \in (C^{-1},C).\]

Recall that our sum is
\[
\sum_{\substack{\frac{1}{2^{s+t} \Nm(\fn) \Psi_2(\fn)} < \Nm(\fs) \Psi_1(\fg \fs) < \frac{T}{\Nm(\fn) \Psi_2(\fn)} \\ g(\fs) = T }} \frac{1}{\Nm(\fs)}.\]

We have
\[
\Nm(\fs) < \frac{T}{\Nm(\fn) \Psi_2(\fn)} \frac{1}{\Psi_1(\fg \fs)} < \frac{T}{\Nm(\fn) \Psi_1(\fn) \Psi_2(\fn)} 
\]
since \(\Psi_1(\fg \fs) \geq \Psi_1(\fn)\), and since
\[
\Psi_1(\fg \fs) < C \Psi_2(\fg \fs) \leq C \Psi_2(\fn) < C^2 \Psi_1(\fn),
\]
we have that
\[
N(\fs) > \frac{1}{C^2 2^{s+t} \Nm(\fn) \Psi_1(\fn) \Psi_2(\fn)}.\]

So our sum can be bounded by
\[
\sum_{\substack{X < \Nm(\fs) \Psi_1(\fg \fs) < Y \\ g(\fs) = T }} \frac{1}{\Nm(\fs)}
\]
where
\[
X = \frac{1}{C^2 2^{s+t} \Nm(\fn) \Psi(\fn)}, \quad Y = \frac{T}{\Nm(\fn) \Psi(\fn)},
\]
and hence by the same approach as before, we once again find that our sum is \(\ll \frac{\log 2T}{T!}\).

Now the entire sum we're interested in is
\[
\ll \sum_{T \in \NN} \frac{(\log 2T)^2}{T!} \sum_{\Nm(\fn) \leq R} \Phi(\fn) \Psi(\fn)  \sum_{\substack{\fg \mid \fn \\ g(\fn/\fg) \leq T}} \frac{1}{\Nm(\fg)}.\]

Finally, we apply Lemma~\ref{lem:divisorsum} to give us a bound of
\[
\ll \sum_{T \in \NN} \frac{(\log 2T)^3}{T!} \sum_{\Nm(\fn) \leq R} \Phi(\fn) \Psi(\fn).\]

Then since 
\[
\sum_{T \in \NN} \frac{(\log 2T)^3}{T!}
\]
converges, our sum is
\[
\ll \sum_{\Nm(\fn) \leq R} \Phi(\fn) \Psi(\fn) \ll \left( \sum_{\Nm(\fn) \leq R} \Phi(\fn) \Psi(\fn) \right)^2
\]
as required, concluding the proof of Theorem~\ref{thm:maintheorem}.


\section{A \(P(\fm,\fn)\) bound in number fields}
\label{sec:Pmn}

In this section, we prove our version of Pollington and Vaughan's \(P(m,n)\) bound.

\begin{proof}[Proof of Proposition~\ref{prop:Pmnbound}.]
We mainly follow the argument on pp 194--196 of \cite{PollingtonVaughan1990}, with some influence from the proof of Lemma 2.8 in \cite{HarmanBook}.

Following the initial steps of the proof of Lemma 3.1 in \cite{Palmer2020}, we define boxes
\[
\Bx(\gamma,\PSI(\fn)) := \bigoplus_{\rho \in \Sigma} B_{\rho}(\rho(\gamma), \psi_{\rho}(\fn)),
\]
and note that
\[
\cA_{\fn} = \left( \bigcup_{\substack{\gamma \in K \\ \dnm(\gamma) = \fn}} \Bx(\gamma, \PSI(\fn)) \right) \cap D_K. \]

Now we can bound \(\lambda(\cA_{\fm} \cap \cA_{\fn})\) by counting the number of pairs of boxes which overlap, and then bounding the measure of the overlap between any two boxes.

For \(\Bx(\beta,\PSI(\fm))\) and \(\Bx(\gamma,\PSI(\fn))\) to overlap, it is necessary that they satisfy
\[
|\rho(\beta) - \rho(\gamma)| \leq 2 \max \{ \psi_{\rho}(\fm), \psi_{\rho}(\fn) \}
\]
for each \(\rho \in \Sigma\), and hence if we write
\[
\Delta_{\rho} := 2 \max \{ \psi_{\rho}(\fm) , \psi_{\rho}(\fn) \},
\]
then we can get an upper bound for the number of boxes which overlap by determining how many \(\beta,\gamma \in D_K\) with \(\dnm \beta = \fm\), \(\dnm \gamma = \fn\) satisfy
\[
|\rho(\beta - \gamma)| < \Delta_{\rho} \hbox{ for all } \rho \in \Sigma.\]

That is to say, we want to evaluate the sum
\begin{equation}
\label{eq:Pmnsum1}
\sum_{\substack{\beta, \gamma \in D_K \\ \dnm \beta = \fm \\ \dnm \gamma = \fn \\ |\rho(\beta - \gamma)| < \Delta_{\rho} \\ \hbox{\scriptsize for all } \rho \in \Sigma}} 1 \end{equation}

By grouping terms of this sum based on the various values that \(\beta - \gamma\) could take, we can reword it as
\[
\sum_{\substack{\theta \in D_K  \\ |\rho(\theta)| < \Delta_{\rho} \\ \hbox{\scriptsize for all } \rho \in \Sigma}} \sum_{\substack{\beta, \gamma \in D_K \\ \dnm \beta = \fm \\ \dnm \gamma = \fn \\ \beta - \gamma = \theta}} 1.
\]

Now for each \(\theta\), we want to know how many ways it can be represented as suitable \(\beta - \gamma\). This relies on the prime factorisation of \(\fm\) and \(\fn\): write
\[
\fm = \prod \fp^{u_{\fp}}, \quad \fn = \prod \fp^{v_{\fp}}. \]

Write \(\fp^n \mid\mid \fa\) if \(\fp^n \mid \fa\) and \(\fp^{n+1} \nmid \fa\). By the Chinese remainder theorem, for each \(\fp^{u_{\fp}} \mid\mid \fm\), there exists some \(\beta_{\fp} \in K\) such that \(\dnm \beta_{\fp} = \fp^{u_{\fp}}\) and 
\[
\sum_{\fp} \beta_{\fp} \equiv \beta \mod \cO_K. \]

Similarly, we can write \(\gamma\) mod \(\cO_K\) as
\[
\sum_{\fp} \gamma_{\fp},
\]
where \(\dnm \gamma_{\fp} = \fp^{v_{\fp}}\). Hence we have
\[
\beta - \gamma \equiv \sum_{\fp} \beta_{\fp} - \gamma_{\fp} \mod \cO_K .\]

Similarly, we can write any \(\theta\) we're aiming to express as
\[
\theta \equiv \sum_{\fp} \theta_{\fp} \mod \cO_K.\]

We can now work one prime ideal at a time, work out how many possibilities there are with regard to this prime ideal, and then combine the possibilities multiplicatively to get an upper bound for the number of \(\theta\) which can be expressed in this way.

For each prime ideal \(\fp\) dividing \(\fm \fn\) there are two possibilities: \(u_{\fp} \neq v_{\fp}\) and \(u_{\fp} = v_{\fp}\). We will work one prime ideal at a time: so in the following, we will suppress the subscripts, and just write \(u\) and \(v\).

If \(u < v\), then we have \(\dnm(\beta_{\fp} - \gamma_{\fp}) = \fp^v\), and for a given \(\theta_{\fp}\) with \(\dnm \theta_{\fp} = \fp^v\), there are \(\Phi(\fp^u)\) choices \((\beta_{\fp}, \gamma_{\fp})\) which represent it. (We can choose whichever \(\beta_{\fp}\) we like, but for that choice of \(\beta_{\fp}\) there is one and only one choice for \(\gamma_{\fp}\).) And by the same argument, if \(u > v\), we have \(\Phi(\fp^v)\) choices.

What if \(u = v\)? In this case, we cannot say that \(\dnm(\beta_{\fp} - \gamma_{\fp}) = \fp^u\), but we can say that \(\dnm(\beta_{\fp} - \gamma_{\fp}) \mid \fp^u\). If in fact we do have \(\dnm \theta_{\fp} = \fp^u\), then there are 
\[
\Nm(\fp^u) \left ( 1 - \frac{2}{\Nm(\fp)} \right)
\]
ways of representing \(\theta_{\fp}\) as \(\beta_{\fp} - \gamma_{\fp}\). However, if \(\dnm \theta_{\fp} = \fp^{u'}\) with \(u' < u\), then we have \(\Phi(\fp^u)\) ways of representing it.

Hence, writing
\[
\fL = \prod_{\substack{\fp \\ u = v}} \fp^u, \quad \fM = \prod_{\substack{\fp \\ u \neq v}} \fp^{\min\{u,v\}}, \quad \fN = \prod_{\substack{\fp \\ u \neq v}} \fp^{\max\{u,v\}}
\]
we find that the sum \eqref{eq:Pmnsum1} is bounded above by
\[
\sum_{\substack{\theta \in D_K \\ \frac{\fL \fN}{\dnm \theta} | \fL \\ |\rho(\theta)| < \Delta_{\rho} \\ \hbox{\scriptsize \ for all } \rho \in \Sigma}} \Phi(\fM) \Nm(\fL) \left( \prod_{\fp | (\fL,\fL \fN (\theta))} \left( 1 - \frac{1}{\Nm(\fp)} \right) \right) \prod_{\substack{\fp \mid \fL \\ \fp \nmid \fL\fN(\theta)}} \left( 1 - \frac{2}{\Nm(\fp)} \right).\]

\textbf{Remark.} Our condition \(\frac{\fL \fN}{\dnm \theta} | \fL\) is equivalent to having
\[
(\theta) = \frac{\fa}{\fL \fN}
\]
where \((\fa,\fN) = 1\). In the classical case (or even the case where \(K\) is a principal ideal domain), we could now just pass to a sum over the numerators of our (not necessarily reduced) fractions: however, this is more delicate in general number fields, and so we leave our sum in terms of elements of \(K\) for now. The term \(\fL\fN(\theta)\), which also appears repeatedly, is the equivalent of this numerator, and corresponds to \(c\) in \S3 of \cite{PollingtonVaughan1990}.

Now, since
\begin{align*}
\prod_{\substack{\fp \mid \fL \\ \fp \nmid \fL\fN(\theta)}} \left( 1 - \frac{2}{\Nm(\fp)} \right) &\leq \prod_{\substack{\fp \mid \fL \\ \fp \nmid \fL\fN(\theta)}} \left( 1 - \frac{1}{\Nm(\fp)} \right)^2  \\
&= \frac{\Phi(\fL)^2}{\Nm(\fL)^2} \prod_{\fp \mid (\fL, \fL \fN(\theta))} \left(1 - \frac{1}{\Nm(\fp)} \right)^{-2}
\end{align*}

we have
\begin{equation}
\label{eq:Pmnsum2}
\lambda(\cA_{\fm} \cap \cA_{\fn}) \ll \bm\delta \frac{\Phi(\fL)^2 \Phi(\fM)}{\Nm(\fL)} \sum_{\substack{\theta \in D_K \\ \frac{\fL \fN}{\dnm \theta} | \fL \\ |\rho(\theta)| < \Delta_{\rho} \\ \hbox{\scriptsize \ for all } \rho \in \Sigma}} \prod_{\fp \mid (\fL, \fL\fN(\theta))} \frac{\Nm(\fp)}{\Nm(\fp) - 1},
\end{equation}
where \(\bm\delta\) is the maximum size of the overlap, representing the worst case where one box is entirely contained within another, and given by
\[
\bm\delta = \prod_{\rho \hbox{\scriptsize \ real}} 2 \min \{\psi_{\rho}(\fm), \psi_{\rho}(\fn) \} \prod_{\rho \hbox{\scriptsize \ complex}} \pi \min \{\psi_{\rho}(\fm), \psi_{\rho}(\fn) \}. \]

The product in \eqref{eq:Pmnsum2} can be rewritten:
\[
\prod_{\substack{\fp | \fL \\ \fp | \fL\fN(\theta)}} \frac{\Nm(\fp)}{\Nm(\fp) - 1} = \prod_{\substack{\fp | \fL \\ \fp | \fL\fN(\theta)}} 1 + \frac{1}{\Nm(\fp) - 1} = \sum_{\substack{\fd | \fL \\ \fd | \fL\fN(\theta) \\ \fd \hbox{\scriptsize \ squarefree}}} \frac{1}{\prod_{\fp \mid \fd} \Nm(\fp) - 1} = \sum_{\substack{\fd | \fL \\ \fd | \fL\fN(\theta)}} \frac{\mu(\fd)^2}{\Phi(\fd)}.\]

Hence \eqref{eq:Pmnsum2} becomes
\[
\lambda(\cA_{\fm} \cap \cA_{\fn}) \ll \bm\delta \frac{\Phi(\fL)^2 \Phi(\fM)}{\Nm(\fL)} \sum_{\substack{\theta \in D_K \\ \frac{\fL \fN}{\dnm \theta} | \fL \\ |\rho(\theta)| < \Delta_{\rho} \\ \hbox{\scriptsize \ for all } \rho \in \Sigma}} \sum_{\substack{\fd | \fL \\ \fd | \fL\fN(\theta)}} \frac{\mu(\fd)^2}{\Phi(\fd)},
\]
where \(\mu\) is the obvious version of the M\"obius function on \(I_K\).

We now want to bound the double sum
\[
\sum_{\substack{\theta \in D_K \\ \frac{\fL \fN}{\dnm \theta} | \fL \\ |\rho(\theta)| < \Delta_{\rho} \\ \hbox{\scriptsize \ for all } \rho \in \Sigma}} \sum_{\substack{\fd \mid \fL \\ \fd \mid \fL \fN(\theta)}} \frac{\mu(\fd)^2}{\Phi(\fd)}.\]

Let \(\hat{\fL}\) denote the ideal of smallest norm in the inverse class of \(\fL\) such that \((\hat{\fL},\fL\fN)=~1\); similarly, let \(\hat{\fN}\) denote the ideal of smallest norm in the inverse class of \(\fN\) such that \((\hat{\fN},\fL\fN) = 1\). Then the ideals \(\fL \hat{\fL}\) and \(\fN \hat{\fN}\) are by definition principal: let \(\lambda\) and \(\nu\) generate these ideals respectively. (For uniqueness, choose the generator so that the absolute value of the sum of the embeddings is minimised.)

Next, using that \((\theta) = \frac{\fa}{\fL\fN} \) for some \(\fa\) such that \((\fa,\fN) = 1\), we can write \(\theta = \frac{\alpha}{\lambda \nu} \) for some \(\alpha \in \hat{\fL} \hat{\fN}\) such that \(((\alpha),\fN) = 1\). So our sum becomes 
\[
\sum_{\substack{\alpha \in \hat{\fL}\hat{\fN} \cap \lambda \nu D_K \\ ((\alpha),\fN) = 1 \\ |\rho(\alpha)| < |\rho(\lambda \nu)| \Delta_{\rho} \\ \hbox{\scriptsize \ for all } \rho \in \Sigma}} \sum_{\substack{\fd \mid \fL \\ \fd \mid (\alpha)\hat{\fL}^{-1}\hat{\fN}^{-1}}} \frac{\mu(\fd)^2}{\Phi(\fd)}.\]

Interchanging the order of the sum, and writing 
\[
\bm\Delta  = \prod_{\rho \in \Sigma} \Delta_{\rho}
\]
we find that this is bounded above by
\[
\sum_{\substack{\fd | \fL \\ \Nm(\fd) < \Nm(\fL \fN) \bm\Delta}} \frac{\mu(\fd)^2}{\Phi(\fd)} \sum_{\substack{\alpha \in \fd\hat{\fL}\hat{\fN} \\ ((\alpha),\fN) = 1 \\ |\rho(\alpha)| < |\rho(\lambda \nu)| \Delta_{\rho} \\ \hbox{\scriptsize \ for all } \rho \in \Sigma}} 1.\]

Now, applying Thm 0 on p102 of \cite{LangANT} from two directions, we find that
\[
\sum_{\substack{\alpha \in \fd\hat{\fL}\hat{\fN} \\ ((\alpha),\fN) = 1 \\ |\rho(\alpha)| < |\rho(\lambda \nu)| \Delta_{\rho} \\ \hbox{\scriptsize \ for all } \rho \in \Sigma}} 1 \ll \Nm\left( \tfrac{\fL \fN}{\fd} \right) \bm\Delta \left(1 - \frac{1}{\Nm(\fN)} \right) \ll \sum_{\substack{\beta \in \cO_K \\ ((\beta),\fN) = 1 \\ |\rho(\beta)| < \Delta_{\rho} \left( \Nm( \frac{\fL \fN}{\fd}) \right)^{1/n} \\ \hbox{\scriptsize \ for all } \rho \in \Sigma}} 1,
\]
and hence we get that our double sum is
\[
\ll \sum_{\substack{\fd | \fL \\ \Nm(\fd) < N(\fL \fN) \bm\Delta}} \sum_{\substack{\beta \in \cO_K \\ ((\beta),\fN) = 1 \\ |\rho(\beta)| < \Delta_{\rho} \left( \Nm( \frac{\fL \fN}{\fd}) \right)^{1/n} \\ \hbox{\scriptsize \ for all } \rho \in \Sigma}} \frac{\mu(\fd)^2}{\Phi(\fd)}.\]

Then, reswapping the sum order and writing \(X_{\rho} = \Delta_{\rho} \Nm(\fL \fN)^{1/n}\), we get
\[
\ll \sum_{\substack{\beta \in \cO_K \\ ((\beta),\fN) = 1 \\ |\rho(\beta)| < X_{\rho} \\ \hbox{\scriptsize \ for all } \rho \in \Sigma}} \sum_{\substack{\fd | \fL \\ \Nm(\fd) < \frac{\Nm(\fL \fN) \bm\Delta}{|\Nm(\beta)|}}}  \frac{\mu(\fd)^2}{\Phi(\fd)}.\]

Then it follows from Theorem~\ref{thm:Merten3} that we have
\[
\sum_{\Nm(\fd) < X} \frac{\mu^2(\fd)}{\Phi(\fd)}\ll \log X,
\]
and hence we find that
\[
\lambda(\cA_{\fm} \cap \cA_{\fn}) \ll \bm\delta \frac{\Phi(\fL)^2 \Phi(\fM)}{\Nm(\fL)} \sum_{\substack{\beta \in \cO_K \\ ((\beta),\fN) = 1 \\ |\rho(\beta)| < X_{\rho}\\ \hbox{\scriptsize \ for all } \rho \in \Sigma}} \log \left( \frac{\Nm(\fL \fN) \bm\Delta}{|\Nm(\beta)|} \right).\]

Next, write
\[
\sum_{\substack{\beta \in \cO_K \\ ((\beta),\fN) = 1 \\ |\rho(\beta)| < X_{\rho} \\ \hbox{\scriptsize \ for all } \rho \in \Sigma}} \log \left( \frac{\Nm(\fL \fN) \bm\Delta}{|\Nm(\beta)|} \right) = \sum_{\substack{\beta \in \cO_K \\ ((\beta),\fN) = 1 \\ |\rho(\beta)| < X_{\rho} \\ \hbox{\scriptsize \ for all } \rho \in \Sigma}} \int_{|\Nm(\beta)|}^{\Nm(\fL \fN) \bm\Delta} \frac{\diff{\theta}}{\theta}.\]

(Note that this \(\theta\) is distinct from the previous use of \(\theta\) in this paper.) This is then equal to
\[
\int_1^{\Nm(\fL \fN) \bm\Delta} \frac{1}{\theta} \bigg(  \sum_{\substack{\beta \in \cO_K \\ ((\beta),\fN) = 1 \\ |\Nm(\beta)| \leq \theta \\ |\rho(\beta)| < X_{\rho} \\ \hbox{\scriptsize \ for all } \rho \in \Sigma}} 1 \bigg) \ \diff{\theta}
\]

So now we want to bound the sum. We have
\[
\sum_{\substack{\beta \in \cO_K \\ ((\beta),\fN) = 1 \\ |\Nm(\beta)| \leq \theta \\ |\rho(\beta)| < X_{\rho} \\ \hbox{\scriptsize \ for all } \rho \in \Sigma}} 1 = \sum_{\substack{\fb \in \cP_K \\ (\fb,\fN) = 1 \\ \Nm(\fb) \leq \theta}} \sum_{\substack{\beta \hbox{\scriptsize \  generates } \fb \\|\rho(\beta)| < X_{\rho} \\ \hbox{\scriptsize \ for all } \rho \in \Sigma}} 1,
\]
where \(\cP_K\) denotes the principal ideals of \(\cO_K\), and we also have that
\[
\sum_{\substack{\beta \hbox{\scriptsize \  generates } \fb \\|\rho(\beta)| < X_{\rho} \\ \hbox{\scriptsize \ for all } \rho \in \Sigma}} 1 \ll \log \left( \frac{\bm\Delta \Nm(\fL \fN)}{\Nm(\fb)} \right)^{\ell},
\]
where \(\ell\) denotes the rank of the group of units of \(\cO_K\). So
\begin{align*}
\sum_{\substack{\beta \in \cO_K \\ ((\beta),\fN) = 1 \\ |\Nm(\beta)| \leq \theta \\ |\rho(\beta)| < \Delta_{\rho} \Nm(\fL \fN)^{1/n} \\ \hbox{\scriptsize \ for all } \rho \in \Sigma}} 1 &\ll  \sum_{1 \leq \xi \leq \theta} \sum_{\substack{\fb \in \cP_K \\ (\fb,\fN) = 1 \\ \Nm(\fb) = \xi}} \log \left( \frac{\bm\Delta \Nm(\fL \fN)}{\Nm(\fb)} \right)^{\ell} \\ &= \sum_{1 \leq \xi \leq \theta} \log \left( \frac{\bm\Delta \Nm(\fL \fN)}{\xi} \right)^{\ell} \sum_{\substack{\fb \in \cP_K \\ (\fb,\fN) = 1 \\ \Nm(\fb) = \xi}} 1 .\end{align*}

Now, write \(X = \bm\Delta \Nm(\fL \fN)\), and let \( \phi(\xi) = \log \left( \frac{X}{\xi} \right)^{\ell} \) and 
\[
a_{\xi} = \# \Big\{ \fb \in \cP_K \ \Big| \ (\fb,\fN) = 1, \ \Nm(\fb) = \xi \Big\}.\] 

Then the sum is equal to \(\sum_{1 \leq \xi \leq \theta} a_{\xi} \phi(\xi)\). Applying Abel summation, we get
\[
\sum_{1 \leq \xi \leq \theta} a_{\xi} \phi(\xi) = \log \left( \frac{X}{\theta} \right)^{\ell} A(\theta) + \ell \int_1^{\theta} \frac{A(t)}{t} \log \left( \frac{X}{t} \right)^{\ell-1} \ \diff{t},
\]
where
\[
A(t) = \sum_{1 \leq \xi \leq t} a_{\xi} = \sum_{\substack{\fb \in \cP_K \\ (\fb,\fN) = 1 \\ \Nm(\fb) \leq t}} 1.\]

We then want to take the integral of this sum divided through by \(\theta\):
\begin{equation}
\label{eq:pairofintegrals}
\int_1^X \log \left( \frac{X}{\theta} \right)^{\ell} \frac{A(\theta)}{\theta} \ \diff{\theta} + \ell \int_1^X \frac{1}{\theta} \int_1^{\theta} \frac{A(t)}{t}  \log \left( \frac{X}{t} \right)^{\ell-1} \ \diff{t} \ \diff{\theta}.\end{equation}

At this point, we want to apply the following lemma, to bound \(A(t)\):

\begin{lemma}
\label{lem:sievetheory}
Let \(\fn \in I_K\). Then
\[
\sum_{\substack{\Nm(\fa) \leq X \\ (\fa,\fn) = 1}} 1 \ll X \prod_{\substack{\fp \mid \fn \\ \Nm(\fp) \leq X}} \left( 1 - \frac{1}{\Nm(\fp)} \right)
\]
\end{lemma}

This lemma will be proven in \S\ref{sec:sievetheory}. Assuming this result, we find that the sum of integrals in \eqref{eq:pairofintegrals} is
\begin{align*}
&\ll \overbrace{\int_1^X\log \left( \frac{X}{\theta} \right)^{\ell} \prod_{\substack{\fp \mid \fN \\ \Nm(\fp) \leq \theta}} \left( 1 - \frac{1}{\Nm(\fp)} \right) \ \diff{\theta}}^{I_1(X,\ell)} \\ &\qquad\qquad  + \ \ \overbrace{\int_1^X \frac{1}{\theta} \int_1^{\theta} \log \left( \frac{X}{t} \right)^{\ell-1} \prod_{\substack{\fp \mid \fN \\ \Nm(\fp) \leq t}} \left( 1 - \frac{1}{\Nm(\fp)} \right)   \ \diff{t} \ \diff{\theta}}^{I_2(X,\ell)}.\end{align*}

Now we bound these two integrals. We split the range of integration in \(I_1(X,\ell)\):
\[
I_1(X,\ell) = \int_1^{\sqrt{X}}\log \left( \frac{X}{\theta} \right)^{\ell} \prod_{\substack{\fp \mid \fN \\ \Nm(\fp) \leq \theta}} \left( 1 - \frac{1}{\Nm(\fp)} \right) \ \diff{\theta}  + \int_{\sqrt{X}}^X\log \left( \frac{X}{\theta} \right)^{\ell} \prod_{\substack{\fp \mid \fN \\ \Nm(\fp) \leq \theta}} \left( 1 - \frac{1}{\Nm(\fp)} \right) \ \diff{\theta} .\]

The integrand in the first integral is at most \(\log(X)^{\ell}\), and hence the integral is bounded by \(\sqrt{X} \log(X)^{\ell} \ll X^{1/2 + \varepsilon}\) for any \(\varepsilon > 0\). In the second integral, the product can be bounded by
\[
\prod_{\substack{\fp \mid \fN \\ \Nm(\fp) \leq \sqrt{X}}} \left( 1 - \frac{1}{\Nm(\fp)} \right),
\]
and hence the integral can be bounded by
\[
\prod_{\substack{\fp \mid \fN \\ \Nm(\fp) \leq \sqrt{X}}} \left( 1 - \frac{1}{\Nm(\fp)} \right) \int_{\sqrt{X}}^X \log \left( \frac{X}{\theta} \right)^{\ell} \ \diff{\theta}.\]

But
\[
\int_{\sqrt{X}}^X \log \left( \frac{X}{\theta} \right)^{\ell} \ \diff{\theta} \leq X \ell!,
\]
and hence we get
\[
\int_{\sqrt{X}}^X\log \left( \frac{X}{\theta} \right)^{\ell} \prod_{\substack{\fp \mid \fN \\ \Nm(\fp) \leq \theta}} \left( 1 - \frac{1}{\Nm(\fp)} \right) \ \diff{\theta} \ll X \prod_{\substack{\fp \mid \fN \\ \Nm(\fp) \leq \sqrt{X}}} \left( 1 - \frac{1}{\Nm(\fp)} \right).\]

Then, by Theorem~\ref{thm:Merten3} we have
\[
\prod_{\substack{\fp \mid \fN \\ \Nm(\fp) \leq \sqrt{X}}} \left( 1 - \frac{1}{\Nm(\fp)} \right) \ll \prod_{\substack{\fp \mid \fN \\ \Nm(\fp) \leq X}} \left( 1 - \frac{1}{\Nm(\fp)} \right)
\]
and
\[
X^{1/2 + \varepsilon} \ll X \prod_{\substack{\fp \mid \fN \\ \Nm(\fp) \leq X}} \left( 1 - \frac{1}{\Nm(\fp)} \right),
\]
and hence it follows that
\[
I_1(X,\ell) \ll X \prod_{\substack{\fp \mid \fN \\ \Nm(\fp) \leq X}} \left( 1 - \frac{1}{\Nm(\fp)} \right).\]

Now we need to tackle
\[
I_2(X,\ell) = \int_1^X \frac{1}{\theta} \int_1^{\theta} \prod_{\substack{\fp \mid \fN \\ \Nm(\fp) \leq t}} \left( 1 - \frac{1}{\Nm(\fp)} \right)  \log \left( \frac{X}{\theta} \right)^{\ell-1} \ \diff{t} \ \diff{\theta}.\]

But this is just
\[
\int_1^X \frac{I_1(\theta,\ell - 1)}{\theta}  \ \diff{\theta} \ll \int_1^X \prod_{\substack{\fp \mid \fN \\ \Nm(\fp) \leq \theta}} \left( 1 - \frac{1}{\Nm(\fp)} \right) \ \diff{\theta} = I_1(X,0).\]

So then we have
\[
\lambda(\cA_{\fm} \cap \cA_{\fn}) \ll \bm\delta \frac{\Phi(\fL)^2 \Phi(\fM)}{\Nm(\fL)}  \Nm(\fL\fN)\bm\Delta  \prod_{\substack{\fp \mid \fN \\ \Nm(\fp) \leq X}} \left( 1 - \frac{1}{\Nm(\fp)} \right).\]

Simplifying a little gives
\[
\lambda(\cA_{\fm} \cap \cA_{\fn}) \ll \bm\delta \bm\Delta \Phi(\fL)^2 \Phi(\fM) \Nm(\fN)  \prod_{\substack{\fp \mid \fN \\ \Nm(\fp) \leq X}} \left( 1 - \frac{1}{\Nm(\fp)} \right).\]

Recall that
\[
N(\fN) \prod_{\fp \mid \fN} \left( 1 - \frac{1}{\Nm(\fp)} \right) = \Phi(\fN),
\]
and hence
\[
\lambda(\cA_{\fm} \cap \cA_{\fn}) \ll \bm\delta \bm\Delta \Phi(\fL)^2 \Phi(\fM) \Phi(\fN) \prod_{\substack{\fp \mid \fN \\ \Nm(\fp) >  \Nm(\fL\fN)\bm\Delta}} \left( 1 - \frac{1}{\Nm(\fp)} \right)^{-1}.\]

Since \( \Phi(\fL)^2 \Phi(\fM) \Phi(\fN) = \Phi(\fm) \Phi(\fn)\), and we have
\[
\left[ \fp \mid \fN \Leftrightarrow \fp \mid \frac{\fm \fn}{(\fm,\fn)^2}  \right], \quad \fL\fN = \frac{\fm \fn}{(\fm,\fn)}, \quad \bm\delta \bm \Delta = 2^s \pi^t \Psi(\fm) \Psi(\fn),
\]
we get
\[
\lambda(\cA_{\fm} \cap \cA_{\fn}) \ll \lambda(\cA_{\fm}) \lambda(\cA_{\fn}) \prod_{\substack{\fp \mid \frac{\fm \fn}{(\fm,\fn)^2} \\ \Nm(\fp) > \bm\Delta \Nm\left(\frac{\fm\fn}{(\fm,\fn)}\right)}} \left(1 - \frac{1}{\Nm(\fp)} \right)^{-1}.\]

So since \(D(\fm,\fn) = \bm\Delta \Nm\left(\frac{\fm\fn}{(\fm,\fn)}\right)\)  we have our result.
\end{proof}

\section{A special case of Selberg's sieve in number fields}
\label{sec:sievetheory}

In this section, we prove Lemma~\ref{lem:sievetheory}. This will be a special case of the Selberg sieve in number fields; however, we believe that providing a full proof of this simple case to be instructive in its own right. We start by setting \(z = X^{\frac{1}{2n+1}}\), and let
\[
\fP = \prod_{\substack{\fp | \fn \\ \Nm(\fp) < z}} \fp. \]

Then, as an auxiliary to studying the sum in question, we study the sum
\[
\sum_{\substack{\Nm(\fa) \leq X \\ (\fa,\fP) = 1}} 1. \]

(Note that since \(\fP | \fn\), we have that \((\fa,\fn) = 1\) implies \((\fa,\fP) = 1\), and hence this sum is an upper bound for the sum in Lemma~\ref{lem:sievetheory}.)

Just as in the standard Selberg sieve, we note that if for each \(\fd \in I_K\) we assign some \(\lambda_{\fd} \in \RR\), then as long as we require \(\lambda_{\cO_K} = 1\), we have
\[
\sum_{\substack{\Nm(\fa) \leq X \\ (\fa,\fP) = 1}} 1 \leq \sum_{\Nm(\fa) \leq X} \left( \sum_{\fd | (\fa,\fP)} \lambda_{\fd} \right)^2. \]

We can then manipulate this sum a little:
\begin{align*}
 \sum_{\substack{\Nm(\fa) \leq X \\ (\fa,\fP) = 1}} 1 &\leq \sum_{\Nm(\fa) \leq X} \left( \sum_{\fd | (\fa,\fP)} \lambda_{\fd} \right)^2 \\ &= \sum_{\Nm(\fa) \leq X} \sum_{\fd_1, \fd_2 | (\fa,\fP)} \lambda_{\fd_1} \lambda_{\fd_2} \\ &= \sum_{\fd_1,\fd_2 | \fP} \lambda_{\fd_1} \lambda_{\fd_2} \sum_{\substack{\Nm(\fa) \leq X \\ \fd_1 | \fa \\ \fd_2 | \fa}} 1.
\end{align*}

By Theorem 5 in \cite{RamMurtyVanOrder2007}, we have
\[
\sum_{\substack{\Nm(\fa) \leq X \\ \fd | \fa}} 1 = \frac{\kappa_K X}{\Nm(\fd)} + R_{\fd},
\]
where
\[
|R_{\fd}| \leq c_K \left(\frac{X}{\Nm(\fd)} \right)^{1-\frac{1}{n}},
\]
and where \(\kappa_K\) and \(c_K\) are constants depending only on \(K\). So writing \(\fD := \lcm(\fd_1,\fd_2)\), we have
\[
 \sum_{\substack{\Nm(\fa) \leq X \\ (\fa,\fP) = 1}} 1 \leq \kappa_K X \sum_{\fd_1,\fd_2 | \fP} \frac{\lambda_{\fd_1} \lambda_{\fd_2}}{\Nm(\fD)} + \sum_{\fd_1,\fd_2 | \fP} |\lambda_{\fd_1} \lambda_{\fd_2} R_{\fD}|.
\]

We write
\[
\Sigma_1 = \sum_{\fd_1,\fd_2 | \fP} \frac{\lambda_{\fd_1} \lambda_{\fd_2}}{\Nm(\fD)} = \sum_{\fd_1,\fd_2 | \fP} \frac{\lambda_{\fd_1} \lambda_{\fd_2}}{\Nm(\fd_1) \Nm(\fd_2)}\Nm(\gcd(\fd_1,\fd_2)),
\]
\[\Sigma_2 = \sum_{\fd_1,\fd_2 | \fP} |\lambda_{\fd_1} \lambda_{\fd_2} R_{\fD}|.\]

We first note that we will choose our \(\lambda_{\fd}\) such that \(\lambda_{\fd} = 0\) for \(\Nm(\fd) \geq z\). Then, using the fact that
\[
\Nm(\fm) = \sum_{\fl | \fm} \Phi(\fl),
\]
we can change the order of summation to give
\begin{align*}
\Sigma_1 &= \sum_{\fd_1,\fd_2 | \fP} \frac{\lambda_{\fd_1} \lambda_{\fd_2}}{\Nm(\fd_1) \Nm(\fd_2)} \sum_{\fl \, \mid \, \gcd(\fd_1,\fd_2)} \Phi(\fl) \\ &= \sum_{\substack{\fl | \fP \\ \Nm(\fl) < z}} \Phi(\fl) \left( \sum_{\substack{\fd | \fP \\ \fl | \fd}}  \frac{\lambda_{\fd}}{\Nm(\fd)} \right)^2.
\end{align*}

Now we actually start setting up the Selberg weights \(\lambda_{\fd}\).

For \(\fv \in I_K\) and \(x > 0\), define
\[
G_{\fv}(x) = \sum_{\substack{\Nm(\fr) < x \\ (\fr,\fv) = 1 \\ \fr \mid \fn}} \frac{\mu^2(\fr)}{\Phi(\fr)},
\]
and let 
\[
G := G_{\cO_K}(z) = \sum_{\substack{\Nm(\fr) < z \\ \fr \mid \fn}} \frac{\mu^2(\fr)}{\Phi(\fr)}.\]

Then for \(\fd \mid \fP\), let
\[
\lambda_{\fd} = \frac{1}{G} \frac{\Nm(\fd) \mu(\fd)}{\Phi(\fd)} G_{\fd} \left( \frac{z}{\Nm(\fd)} \right). \]

(Note that with these choices, we have \(\lambda_{\fd} = 0\) for \(\Nm(\fd) \geq z\) as assumed, since \(G_{\fv}(x) = 0\) for \(x \leq 1\).)

We claim that with these choices we have \(\Sigma_1 = \frac{1}{G}\). Recall that
\[
\Sigma_1 = \sum_{\substack{\fl | \fP \\ \Nm(\fl) < z}} \Phi(\fl) \left( \sum_{\substack{\fd | \fP \\ \fl | \fd}}  \frac{\lambda_{\fd}}{\Nm(\fd)} \right)^2,
\]
and consider the sum inside the square. We have
\begin{align*}
\sum_{\substack{\fd | \fP \\ \fl | \fd}} \frac{\lambda_{\fd}}{\Nm(\fd)}  &= \frac{1}{G} \sum_{\substack{\fd | \fP \\ \fl | \fd}} \frac{\mu(\fd)}{\Phi(\fd)} G_{\fd} \left( \frac{z}{\Nm(\fd)} \right) \\ &= \frac{1}{G} \sum_{\substack{\fm | \fP \\ (\fm,\fl) = 1}} \frac{\mu(\fl \fm)}{\Phi(\fl \fm)} G_{\fl \fm} \left( \frac{z}{\Nm(\fl \fm)} \right) \\ &= \frac{\mu(\fl)}{\Phi(\fl) G} \sum_{\substack{\fm | \fP \\ (\fm,\fl) = 1}} \frac{\mu(\fm)}{\Phi(\fm)} \sum_{\substack{\Nm(\fr) < \frac{z}{\Nm(\fl\fm)} \\ (\fr,\fl\fm) = 1 \\ \fr | \fn}} \frac{\mu^2(\fr)}{\Phi(\fr)}.
\end{align*}

We claim that the double sum is just equal to \(1\), and hence that this expression is equal to \(\frac{\mu(\fl)}{\Phi(\fl) G}\). To see this, combine the two sums and then exchange the order (writing \(\ff = \fm \fr\)):
\begin{align*}
\sum_{\substack{\fm | \fP \\ (\fm,\fl) = 1}} \frac{\mu(\fm)}{\Phi(\fm)} \sum_{\substack{\Nm(\fr) < \frac{z}{\Nm(\fl\fm)} \\ (\fr,\fl\fm) = 1 \\ \fr | \fn}} \frac{\mu^2(\fr)}{\Phi(\fr)} &= \sum_{\substack{\fm | \fP \\ (\fm,\fl) = 1}} \sum_{\substack{\Nm(\fr) < \frac{z}{\Nm(\fl\fm)} \\ (\fr,\fl\fm) = 1 \\ \fr | \fn}} \mu(\fm) \frac{\mu^2(\fm \fr)}{\Phi(\fm \fr)} 
\\ &= \sum_{\substack{\Nm(\ff) < \frac{z}{\Nm(\fl)} \\ (\ff,\fl) = 1 \\ \ff \mid \fn}} \frac{\mu^2(\ff)}{\Phi(\ff)} \sum_{\fm \mid \ff} \mu(\fm).
\end{align*}

Then since
\[
\sum_{\fd \mid \fn} \mu(\fd) = \left\{ \begin{array}{rl} 1 & \hbox{if }\fn = \cO_K, \\ 0 & \hbox{otherwise,} \end{array} \right.
\]
we have that the only non-zero term in the outer sum is \(\ff = \cO_K\), giving a value of \(1\) as required. Hence
\[
\sum_{\substack{\fd | \fP \\ \fl | \fd}} \frac{\lambda_{\fd}}{\Nm(\fd)} = \frac{\mu(\fl)}{\Phi(\fl) G},
\]
and therefore 
\[
\Sigma_1 = \sum_{\substack{\fl | \fP \\ \Nm(\fl) < z}} \Phi(\fl) \left( \frac{\mu(\fl)}{\Phi(\fl) G} \right)^2 =  \frac{1}{G^2} \sum_{\substack{\fl \mid \fP \\ \Nm(\fl) < z}} \frac{\mu(\fl)^2}{\Phi(\fl)}.\]

Then since the squarefree divisors of \(\fP\) with norm \(< z\) are exactly the squarefree divisors of \(\fn\) with norm \(< z\), the sum is just equal to \(G\), and hence
\[
\Sigma_1 = \frac{1}{G}.\]

We now put a lower bound on \(G\) (and hence an upper bound on \(\Sigma_1\)). Setting
\[
\fk = \prod_{\substack{\Nm(\fp) < z \\ \fp \nmid \fn}} \fp,
\]
we have
\[
G = \sum_{\substack{\Nm(\fr) < z \\ (\fr,\fk) = 1}} \frac{\mu^2(\fr)}{\Phi(\fr)}.\]

Now, note that we have
\begin{align*}
\sum_{\Nm(\fr) < z} \frac{\mu^2(\fr)}{\Phi(\fr)} &= \sum_{\fl \mid \fk} \sum_{\substack{\Nm(\fr) < z \\ (\fr,\fk) = \fl}} \frac{\mu^2(\fr)}{\Phi(\fr)} \\
&= \sum_{\fl \mid \fk} \sum_{\substack{\Nm(\fh) < \frac{z}{\Nm(\fl)} \\ (\fh,\fk/\fl) = 1 \\ (\fh,\fl) = 1}} \frac{\mu^2(\fl \fh)}{\Phi(\fl \fh)} \\ &= \sum_{\fl \mid \fk} \frac{\mu^2(\fl)}{\Phi(\fl)} \sum_{\substack{\Nm(\fr) < \frac{z}{\Nm(\fl)} \\ (\fr,\fk) = 1}} \frac{\mu^2(\fr)}{\Phi(\fr)} \\
&\leq \sum_{\fl \mid \fk} \frac{\mu^2(\fl)}{\Phi(\fl)} \sum_{\substack{\Nm(\fr) < z \\ (\fr,\fk) = 1}} \frac{\mu^2(\fr)}{\Phi(\fr)},
\end{align*}
and hence
\[
G \geq \left( \sum_{\fl \mid \fk} \frac{\mu^2(\fl)}{\Phi(\fl)} \right)^{-1} \sum_{\Nm(\fr) < z} \frac{\mu^2(\fr)}{\Phi(\fr)}.\]

Defining \(\kappa(\fn)\) to be the squarefree divisor of \(\fn\) of largest norm, we have
\[
\sum_{\Nm(\fr) < z} \frac{\mu^2(\fr)}{\Phi(\fr)} = \sum_{\substack{\fn \in I_K \\ \Nm(\kappa(\fn)) < z}} \frac{1}{\Nm(\fn)} > \sum_{\Nm(\fn) < z} \frac{1}{\Nm(\fn)} \gg \log z.\]

We also have
\[
\left( \sum_{\fl \mid \fk} \frac{\mu^2(\fl)}{\Phi(\fl)} \right)^{-1} = \prod_{\fp \mid \fk} \left(1 + \frac{1}{\Nm(\fp) - 1} \right)^{-1} = \prod_{\fp \mid \fk} \left(1 - \frac{1}{\Nm(\fp)} \right),
\]
giving
\[
\Sigma_1 \ll \prod_{\fp \mid \fk} \left(1 - \frac{1}{\Nm(\fp)} \right)^{-1} \frac{1}{\log z}.\]

Now noting that we have that \(\fp \mid \fk \Leftrightarrow \fp \nmid \fn\) for \(\Nm(\fp) < z\) by the definition of \(\fk\), we have
\[
\Sigma_1 \ll \prod_{\substack{\Nm(\fp) < z \\ \fp \nmid \fn}} \left( 1 - \frac{1}{\Nm(\fp)} \right)^{-1} \frac{1}{\log z}.\]

We can rewrite
\[
\prod_{\substack{\Nm(\fp) < z \\ \fp \nmid \fn}} \left( 1 - \frac{1}{\Nm(\fp)} \right)^{-1} \frac{1}{\log z} = \left( \log z \prod_{\Nm(\fp) < z} \left( 1 - \frac{1}{\Nm(\fp)} \right)\right)^{-1}  \prod_{\substack{\Nm(\fp) < z \\ \fp \mid \fn}} \left( 1 - \frac{1}{\Nm(\fp)} \right).\]

But then the first term on the RHS is bounded by Theorem~\ref{thm:Merten3}. So we get
\[
\Sigma_1 \ll \prod_{\substack{\Nm(\fp) < z \\ \fp \mid \fn}} \left( 1 - \frac{1}{\Nm(\fp)} \right).\]

Now we can tackle \(\Sigma_2\). Recall that we defined
\[
\Sigma_2 = \sum_{\fd_1, \fd_2 | \fP} |\lambda_{\fd_1} \lambda_{\fd_2} R_{\fD}|,
\]
where 
\[|R_{\fD}| \leq c_k \left(\frac{X}{\Nm(\fD)} \right)^{1-\frac{1}{n}} \aand \lambda_{\fd} = \frac{1}{G} \frac{\Nm(\fd) \mu(\fd)}{\Phi(\fd)} G_{\fd} \left( \frac{z}{\Nm(\fd)} \right).\]

First, we'll show that each of the \(\lambda_{\fd}\) satisfy \(|\lambda_{\fd}| \leq 1\). We analyse the sum \(G\) considering the factors the terms share with \(\fd\):
\[
G = \sum_{\substack{\Nm(\fr) < z \\ \fr \mid \fn}} \frac{\mu^2(\fr)}{\Phi(\fr)} = \sum_{\fl \mid \fd} \sum_{\substack{\Nm(\fr) < z \\ \fr \mid \fn \\ (\fr,\fd) = \fl}} \frac{\mu^2(\fr)}{\Phi(\fr)}.\]

Decompose \(\fr = \fl \fh\): we get
\[
G = \sum_{\fl | \fd} \sum_{\substack{\Nm(\fh) < \frac{z}{\Nm(\fl)} \\ (\fh,\fd/\fl) = 1 \\ (\fh,\fl) = 1 \\ \fh | \fn}} \frac{\mu^2(\fl \fh)}{\Phi(\fl \fh)} = \sum_{\fl | \fd} \frac{\mu^2(\fl)}{\Phi(\fl)} \sum_{\substack{\Nm(\fh) < \frac{z}{\Nm(\fl)} \\ (\fh,\fd/\fl) = 1 \\ (\fh,\fl) = 1 \\ \fh | \fn}} \frac{\mu^2(\fh)}{\Phi(\fh)}.\]

Now, the innermost sum is just equal to \(G_{\fd}(\frac{z}{\Nm(\fl)})\), and since \(\fl \mid \fd\), we have \(\Nm(\fl) \leq \Nm(\fd)\), and hence
\[
G_{\fd} \left( \frac{z}{\Nm(\fl)} \right) \geq G_{\fd} \left( \frac{z}{\Nm(\fd)} \right).\]

So
\[
G = \sum_{\fl \mid \fd} \frac{\mu^2(\fl)}{\Phi(\fl)} G_{\fd} \left( \frac{z}{\Nm(\fl)} \right) \geq G_{\fd} \left( \frac{z}{\Nm(\fd)} \right) \sum_{\fl \mid \fd} \frac{\mu^2(\fl)}{\Phi(\fl)}.\]

Finally, we have
\[
\sum_{\fl \mid \fd} \frac{\mu^2(\fl)}{\Phi(\fl)} = \prod_{\fp \mid \fd} \left( 1 + \frac{1}{\Nm(\fp) - 1} \right) = \prod_{\fp \mid \fd} \left( 1 - \frac{1}{\Nm(\fp)} \right)^{-1},
\]
giving us
\[
G \geq \frac{\Nm(\fd)}{\Phi(\fd)}  G_{\fd} \left( \frac{z}{\Nm(\fd)} \right).\]

So by our definition of \(\lambda_{\fd}\), we have \(|\lambda_{\fd}| \leq 1\). Now, since we also have \(\lambda_{\fd} = 0\) for \(\Nm(\fd) \geq z\), we get
\[
\Sigma_2 \leq \sum_{\substack{\Nm(\fd_1) < z \\ \Nm(\fd_2) < z}} |R_{\fD}| \leq c_K X^{1-\frac{1}{n}} \sum_{\substack{\Nm(\fd_1) < z \\ \Nm(\fd_2) < z}}  \frac{1}{\Nm(\fD)^{1-\frac{1}{n}}} \leq c_K X^{1-\frac{1}{n}} \left( \sum_{\Nm(\fa) < z} 1 \right)^2.
\]

Then, since we know that
\[
\sum_{\Nm(\fa) < z} 1 \ll z,
\]
we get
\[
\Sigma_2 \ll X^{1 - \frac{1}{n}} z^2 = X^{1-\frac{1}{n(2n+1)}}.\]

So to recap, we have
\[
 \sum_{\substack{\Nm(\fa) \leq X \\ (\fa,\fn) = 1}} 1 \leq \sum_{\substack{\Nm(\fa) \leq X \\ (\fa,\fP) = 1}} 1 \leq \kappa_K X \Sigma_1 + \Sigma_2,
\]
where
\[
\Sigma_1 \ll \prod_{\substack{\Nm(\fp) < X^{\frac{1}{2n+1}} \\ \fp \mid \fn}} \left( 1 - \frac{1}{\Nm(\fp)} \right), \quad \Sigma_2 \ll X^{1-\frac{1}{n(2n+1)}}.\]

Then, since we know that
\[
 \prod_{\Nm(\fp) < y} \left( 1 - \frac{1}{\Nm(\fp)} \right) \asymp \frac{1}{\log y},
\]
we find that
\[
 \prod_{\Nm(\fp) < X^{\frac{1}{2n+1}}} \left( 1 - \frac{1}{\Nm(\fp)} \right) \ll \frac{1}{\log X^{\frac{1}{2n+1}}} = \frac{2n+1}{\log X} \ll \prod_{\Nm(\fp) < X} \left( 1 - \frac{1}{\Nm(\fp)} \right),
\]
and hence 
\begin{align*}
\prod_{\substack{\Nm(\fp) < X^{\frac{1}{2n+1}} \\ \fp \mid \fn}} \left( 1 - \frac{1}{\Nm(\fp)} \right) &= \prod_{\Nm(\fp) < X^{\frac{1}{2n+1}}} \left( 1 - \frac{1}{\Nm(\fp)} \right) \prod_{\substack{\Nm(\fp) < X^{\frac{1}{2n+1}} \\ \fp \nmid \fn}} \left( 1 - \frac{1}{\Nm(\fp)} \right)^{-1} \\
 &\ll \prod_{\Nm(\fp) < X} \left( 1 - \frac{1}{\Nm(\fp)} \right) \prod_{\substack{\Nm(\fp) < X \\ \fp \nmid \fn}} \left( 1 - \frac{1}{\Nm(\fp)} \right)^{-1} \\
 &= \prod_{\substack{\Nm(\fp) < X \\ \fp \mid \fn}} \left( 1 - \frac{1}{\Nm(\fp)} \right).
\end{align*}

Then since we clearly have
\[
X^{1-\frac{1}{n(2n+1)}} \ll X \prod_{\substack{\Nm(\fp) < X \\ \fp \mid \fn}} \left( 1 - \frac{1}{\Nm(\fp)} \right),
\]
we have
\[
\sum_{\substack{\Nm(\fa) \leq X \\ (\fa,\fn) = 1}} 1 \ll X \prod_{\substack{\Nm(\fp) < X \\ \fp \mid \fn}} \left( 1 - \frac{1}{\Nm(\fp)} \right)
\]
as required.

\section{Erd\H{o}s's \(g\)-function in number fields}
\label{sec:NFgfct}
In this section, we show that the analogue of Erd\H{o}s's \(g\)-function in number fields which we defined in Definition~\ref{def:gfct} satisfies some key properties.

\begin{lemma}
\label{lem:basicpropertiesofg}
We have
\[
\prod_{\substack{\fp | \fn \\ \Nm(\fp) > g(\fn)}} \left(1 - \frac{1}{\Nm(\fp)} \right)^{-1} \ll 1 \aand \frac{\Nm(\fn)}{\Phi(\fn)}\ll \log(2 g(\fn)). \]
\end{lemma}

\begin{proof}
Observe that 
\begin{align*}
\prod_{\substack{\fp | \fn \\ \Nm(\fp) > g(\fn)}} \left(1 - \frac{1}{\Nm(\fp)} \right)^{-1}
&= \exp \left( \sum_{\substack{\fp | \fn \\ \Nm(\fp) > g(\fn)}} - \log \left( 1 - \frac{1}{\Nm(\fp)} \right) \right)  \\
&\leq \exp \left( \sum_{\substack{\fp | \fn \\ \Nm(\fp) > g(\fn)}} \frac{1}{\Nm(\fp)} + \sum_{\fp} \sum_{j=2}^{\infty} \frac{1}{\Nm(\fp)^j j} \right).
\end{align*}

Then the first sum inside the exponential is bounded by the definition of \(g\), and we have
\[
\sum_{\fp} \sum_{j=2}^{\infty} \frac{1}{\Nm(\fp)^j j} < \sum_{\fp} \sum_{j=2}^{\infty} \frac{1}{\Nm(\fp)^j} \leq 2 \zeta_K(2),
\]
giving our first result.

For the second part, we observe that
\[
\frac{\Nm(\fn)}{\Phi(\fn)} = \left( \prod_{\substack{\fp | \fn \\ \Nm(\fp) \leq g(\fn)}} \left( 1 - \frac{1}{\Nm(\fp)} \right) \right)^{-1} \prod_{\substack{\fp | \fn \\ \Nm(\fp) > g(\fn)}} \left( 1 - \frac{1}{\Nm(\fp)} \right)^{-1}. \]

Then the second product is \(\ll 1\), by the first part, and by Theorem~\ref{thm:Merten3}, we have that 
\[
\prod_{\substack{\fp | \fn \\ \Nm(\fp) \leq g(\fn)}} \left( 1 - \frac{1}{\Nm(\fp)} \right) \gg \log(2 g(\fn))^{-1},
\]
giving our second result.
\end{proof}

\begin{proof}[Proof of Lemma~\ref{lem:glemma1}]
We follow closely the proof of Lemma 2.9 in \cite{HarmanBook} (which in turn follows the proof in \cite{Erdos1970}). Define \(E = e^{t^2}\), and sort the ideals \(\fn\) with \(\Nm(\fn) \leq X\) into two disjoint sets:
\begin{itemize}
  \item those with at least \(2t\) distinct prime factors in \([t,E)\);
  \item those with less than \(2t\) distinct prime factors in \([t,E)\).
\end{itemize}

We know that the number of ideals with \(\Nm(\fn) \leq X\) is \(\ll X\) (where the implicit constant depends only on the number field \(K\)), and hence the number of ideals \(\fn\) in the first case is
\[
\ll \frac{X}{(2t)!} \left( \sum_{t \leq \Nm(\fp) < E} \frac{1}{\Nm(\fp)} \right)^{2t}. \]

By Theorem~\ref{thm:Merten2}, we have that
\[
\sum_{t \leq \Nm(\fp) < E} \frac{1}{\Nm(\fp)} \ll \log \log E \ll \log 2t,
\]
and hence the number of ideals \(\fn\) in the first case is
\[
 \ll X \frac{(A \log 2t)^{2t}}{(2t)!} = \frac{X}{t!} \left( \frac{(A \log 2t)^2}{t+1} \right) \cdots \left( \frac{(A \log 2t)^2}{2t} \right). \]
 
Then since each term in the product after the first is \(\leq 1\) for sufficiently large \(t\), we have that the size of the first set is \(\ll \frac{X}{t!}\).

So now consider the second case. Let \(\Nm(\fp_1), \Nm(\fp_2), \ldots, \Nm(\fp_{2t})\) be the smallest \(2t\) norms you get from distinct prime ideals whose norms exceed \(t\). For large \(t\), we have
\[
\sum_{j=1}^{2t} \frac{1}{\Nm(\fp_j)} < \frac{1}{4},
\]
as a result of Landau's prime ideal theorem. This implies that for \(\fn\) to have \(g(\fn) \geq t\) but \textit{not} have \(2t\) distinct prime factors whose norms lie in \([t,E)\), we need 
\[
\sum_{\substack{\fp | \fn \\ \Nm(\fp) > E}} \frac{1}{\Nm(\fp)} > \frac{1}{4},
\]
or rewritten,
\[
1 < 4 \sum_{\substack{\fp | \fn \\ \Nm(\fp) > E}} \frac{1}{\Nm(\fp)}. \]

Then
\[
\sum_{\substack{\Nm(\fn) \leq X \\ g(\fn) \geq t \\ \hbox{\scriptsize $\fn$ has $<2t$} \\ \hbox{\scriptsize distinct prime} \\ \hbox{\scriptsize factors in $[t,E)$}}} 1 < 4 \sum_{\substack{\Nm(\fn) \leq X \\ g(\fn) \geq t}} \sum_{\substack{\fp | \fn \\ \Nm(\fp) > E}} \frac{1}{\Nm(\fp)} < 4 \sum_{\Nm(\fn) \leq X} \sum_{\substack{\fp | \fn \\ \Nm(\fp) > E}} \frac{1}{\Nm(\fp)}.\]

The last double sum can be re-expressed as
\[
4 \sum_{\Nm(\fp) > E} \frac{ \# \{ \Nm(\fn) \leq X \ | \  \fp | \fn \} }{\Nm(\fp)} \ll X \sum_{\Nm(\fp) > E} \frac{1}{\Nm(\fp)^2} .\]

Since there can be at most \(n\) prime ideals of any given norm, we can get a bound on this sum (which is inefficient but sufficient for our purposes) by
\[
nX \sum_{m > E} \frac{1}{m^2} \ll \frac{X}{E}. \]

Then since \(E = e^{t^2} \gg t!\) (by taking logs and using Stirling's formula), we have \(\frac{X}{E} \ll \frac{X}{t!}\), and hence we have our result.
\end{proof}

\begin{proof}[Proof of Lemma~\ref{lem:divisorsum}]
Firstly, rewriting the second part of Lemma~\ref{lem:basicpropertiesofg}, we have
\[
\frac{\Phi(\fn)}{\Nm(\fn)} \log(2 g (\fn)) \gg 1
\]
for \(\fn \in I_K\), and hence taking \(\fn = \frac{\fq}{\fd}\), we get
\[
\frac{\Phi( \frac{\fq}{\fd})}{\Nm( \frac{\fq}{\fd})} \log(2 g (\tfrac{\fq}{\fd})) \gg 1. \]

If \(t \geq g (\tfrac{\fq}{\fd})\), then this gives us
\[
1 \ll \frac{\Phi( \frac{\fq}{\fd})}{\Nm( \frac{\fq}{\fd})} \log 2t. \]

Hence we have
\[
\sum_{\substack{\fd | \fq \\ g(\frac{\fq}{\fd}) \leq t}} \frac{1}{\Nm(\fd)} \ll \sum_{\substack{\fd | \fq \\ g(\frac{\fq}{\fd}) \leq t}} \frac{1}{\Nm(\fd)} \frac{\Phi( \frac{\fq}{\fd})}{\Nm( \frac{\fq}{\fd})} \log 2t = \log 2t \left( \frac{\sum_{\fr | \fq} \Phi(\fr)}{\Nm(\fq)} \right) = \log 2t.\]
\end{proof}


\begin{thebibliography}{99}




\bibitem{Chen2015}
Z. Chen. On metric theory of Diophantine approximation for complex numbers. \textit{Acta Arithmetica}, 170:27--46, 2015.

%
\bibitem{Dirichlet1842}
L.G.P. Dirichlet. Verallgemeinerung eines Satzes aus der Lehre von den Kettenbr\"{u}chen nebst einigen Anwendungen auf die Theorie der Zahlen. \textit{S.-B. Preuss. Akad. Wiss.}, 93--95, 1842.

\bibitem{DuffinSchaeffer1941}
R.J. Duffin and A.C. Schaeffer. Khintchine's problem in metric Diophantine approximation. \textit{Duke Mathematical Journal}, 8:243--255, 1941.

\bibitem{Erdos1970}
P. Erdos. On the distribution of the convergents of almost all real numbers. \textit{J. Number Theory}, 2:425-441, 1970.


\bibitem{HalberstamRichertBook}
H. Halberstam and H. Richert. \textit{Sieve methods}. Dover Publications Inc., 2011.

\bibitem{HarmanBook}
G. Harman. \textit{Metric number theory}. Oxford University Press, 1998.
%
%
\bibitem{Hermite1850}
C. Hermite. Sur l'introduction des variables continues dans la th\'{e}orie des nombres. \textit{J. Reine Angew. Math.}, 41:191--216, 1850.

\bibitem{Hurwitz1891}
A. Hurwitz. Ueber die angen\"{a}herte Darstellung der Irrationalzahlen durch rationale Br\"{u}che. \textit{Math. Ann.} 39(2):279--284, 1891.

\bibitem{Khinchin1924}
A. Khintchine. Einige S\"{a}tze \"{u}ber Kettenbr\"{u}che, mit Anwendungen auf die Theorie der Diophantischen Approximationen. \textit{Math. Ann.} 92:115--125, 1924.

\bibitem{KoukoulopoulosMaynard2020}
D. Koukoulopoulos and J. Maynard. On the Duffin--Schaeffer conjecture. \textit{Annals of Mathematics} 192:251--307, 2020.

\bibitem{LangANT}
S. Lang. \textit{Algebraic number theory.} Springer-Verlag New York Inc., 1986.
%
%


\bibitem{Mertens1874}
F. Mertens. Ein Beitrag zur analytischen Zahlentheorie. \textit{J. Reine Angew. Math}, 78:46--62, 1874.


\bibitem{NakadaWagner1991}
H. Nakada and G. Wagner. Duffin--Schaeffer theorem of diophantine approximation for complex numbers. \textit{Ast\'{e}risque}, 198--200:259--263, 1991.

\bibitem{Palmer2020}
M. Palmer. The Duffin--Schaeffer theorem in number fields. \textit{Acta Arithmetica}, 196(1): 1--16, 2020.

\bibitem{PollingtonVaughan1990}
A.D. Pollington and R.C. Vaughan. The $k$-dimensional Duffin and Schaeffer conjecture. \textit{Mathematika}, 37:190--200, 1990.
%
%

\bibitem{RamMurtyVanOrder2007}
M. Ram Murty and Jeanine Van Order. Counting integral ideals in a number field. \textit{Expo. Math.} 25:53--66, 2007.

\bibitem{Rosen1999}
M. Rosen. A generalization of Mertens' theorem. \textit{J. Ramanujan Math. Soc.} 14(1):1--19, 1999. 


\bibitem{Vaaler1978}
J.D. Vaaler. On the metric theory of Diophantine approximation. \textit{Pacific J. Math.}, 76:527--539, 1978.

\end{thebibliography}
\end{document}